\documentclass[12pt,reqno]{amsart}
\usepackage{latexsym,amsmath,amsfonts,amscd,amssymb}
\usepackage{lscape}
\usepackage{array}
\usepackage[mathscr]{eucal} 
\usepackage{mathrsfs}
\usepackage{graphicx}
\usepackage{xypic}
\usepackage[T1]{fontenc}
\usepackage[colorlinks=true,citecolor=blue,urlcolor=blue,linkcolor=blue]{hyperref}
\theoremstyle{plain}
\newtheorem{theorem}{Theorem}[section]

\newtheorem*{theorem*}{Theorem}
\newtheorem{lemma}[theorem]{Lemma}

\theoremstyle{definition}

\numberwithin{equation}{section}

\begin{document}

\title{Stability Of The Parabolic Poincar\'e Bundle}

\author[S. Basu]{Suratno Basu}
\address{Institute of Mathematical Sciences, HBNI, CIT Campus, Tharamani, Chennai 600113, India}
\email{suratnob@imsc.res.in}

\author[I. Biswas]{Indranil Biswas}
\address{School of Mathematics, Tata Institute of Fundamental Research, Homi Bhabha Road, Mumbai 
4000005, India}
\email{indranil@math.tifr.res.in}

\author[K. Dan]{Krishanu Dan}
\address{Chennai Mathematical Institute, H1 Sipcot IT Park, 
Siruseri, Kelambakkam - 603103, INDIA.}
\email{krishanud@cmi.ac.in}



\begin{abstract}
Given a compact Riemann surface $X$ and a moduli space $M_{\alpha}(\Lambda)$ of
parabolic stable bundles on it of fixed determinant of complete parabolic flags,
we prove that the Poincar\'e parabolic bundle on $X\times M_{\alpha}(\Lambda)$ is
parabolic stable with respect to a natural polarization on $X\times M_{\alpha}(\Lambda)$.
\end{abstract}

\maketitle

\section{Introduction}

Let $X$ be a smooth irreducible complex projective curve of genus $g$, with 
$g\,\geq\, 2$. Fix an integer $r\,\geq\, 2$ and a line bundle $L$ over $X$ of 
degree $d$ such that $r$ and $d$ are coprime. Let $M_{r,L}$ denote the moduli 
space of isomorphism classes of stable bundles of rank $r$ with 
$\wedge^{r}E\,\simeq\, L$. A vector bundle $V$ over $X\times M_{r,L}$ is called a 
{\it Poincar\'e bundle} if its restriction $V\vert_{X\times \{[E]\}}$ is 
isomorphic to $E$ for all closed points $[E]\,\in\, M_{r,L}$. It is known that a 
Poincar\'e bundle exists; moreover, any two of them differ by tensoring with a 
line bundle pulled back from $M_{r,L}$. Balaji, Brambila-Paz and Newstead proved 
in \cite{BBN} that any such Poincar\'e bundle is stable with respect to any ample 
divisor in $X \times M_{r,L}$. Recently, Biswas, Gomez and Hoffman studied in 
\cite{BGH} the similar question for the moduli space of principal $G$-bundles.

In this short note we consider certain moduli spaces of stable parabolic bundles on 
$X$. Let us fix a finite set $D$ of $n$ closed points in $X$. We denote by 
$M_{\alpha}(\Lambda)$ the moduli space of stable parabolic bundles of rank $r$ on 
$X$ with fixed determinant $\Lambda$ having full flags at each points of $D$ and 
rational parabolic weights $\alpha\,=\,\{\alpha_{j}^i\}$, $1\,\leq\, j \,\leq\, r$ 
and $1 \,\leq\, i \,\leq\, n$. In this case it is known that there 
exists a vector bundle $\mathcal{U}_{\alpha}$ over $X\times M_{\alpha}(\Lambda)$ 
which has a natural parabolic structure over the divisor $D\times 
M_{\alpha}(\Lambda)$, and moreover, its restriction to each closed points $[E_*]$ 
is isomorphic to $E_*$ as parabolic bundle \cite{BD}. Any two such bundles differ 
by tensoring with a line bundle pulled back from $M_{\alpha}(\Lambda)$. We call 
such bundles {\it Poincar\'e parabolic bundles}. So, it is natural to ask whether 
this bundles are parabolic (slope) stable.

We prove the following:

\begin{theorem}
Let $\mathcal{U}$ be a Poincar\'e parabolic bundle over $X\times M_{\alpha}(\Lambda)$. 
Then $\mathcal{U}$ is a parabolic (slope) stable bundle with respect to a
natural ample divisor.
\end{theorem}

We adopt the strategy of proof in \cite{BBN} in the given context.
 
\section{Preliminaries}

Let $X$ be an irreducible smooth complex projective curve of genus $g \,\geq\, 2$.
Fix $n$ distinct points $x_{1},\, \cdots\, ,x_{n}$ on $X$, and denote the
divisor $x_{1}+\cdots +x_{n}$ on $X$ by $D$. Let $E$ be a holomorphic vector bundle on $X$ of
rank $r$.

A {\it quasi-parabolic structure} over $E$ is a strictly decreasing
filtration of linear subspaces 
\[E_{_{x_i}}=F_i^1 \supset F_i^2 \supset \cdots \supset F_i^{k_i}\supset
F_i^{{k_i}+1} = 0
\]
for every $x_i\in D$. We set \[r_j^i \,:= \,\mbox{dim}F_i^j - \mbox{dim}F_i^{j+1}\, .\] 
The integer $k_i$ is called the {\it length} of the flag and the sequence $(r_1^{i}, 
r_2^{i},\cdots, r_{k_i}^{i})$ is called the {\it type} of the flag at $x_i$. A {\it 
parabolic structure} on $E$ over the divisor $D$ is a quasi-parabolic structure
as above together with a sequence of real numbers
\[0\leq \alpha_1^i < \alpha_2^i < \cdots < \alpha_{k_i}^i < 1.\] 
The {\it parabolic degree} of $E$ is defined to be 
\[\mbox{par-deg}(E):=\mbox{deg}(E)+\sum_{x_i\in D}\sum_{j=1}^{k_i}\alpha_j^ir_j^i\] 
and the {\it parabolic slope} of $E$ is
\[\mbox{par-}\mu(E) :=\frac{\mbox{par-deg}(E)}{r}.\]
(See \cite{MS}.)

For any subbundle $F\subseteq E$, there exists an induced parabolic structure on $F$
whose quasi-parabolic filtration over $x_i$ is given by the distinct subspaces in
\[F_{x_i}=F_i^1 \cap F_{x_i} \supset F_i^2 \cap F_{x_i} \supset \cdots \supset F_i^{k_0}\cap F_{x_i}\supset 0\, ,\] 
where $k_0 := {\rm max}\{j\in \{1,\cdots, k_i\} \mid F_i^j\cap F_{x_i}\neq 0\}$;
the parabolic weight of $F_i^j \cap F_{x_i}$ is the maximum of all $\alpha_i^\ell$ such that
$F_i^j \cap F_{x_i}\,=\, F_i^\ell \cap F_{x_i}$.

A parabolic vector bundle $E$ with parabolic structure over $D$ is said to be
{\it stable} (respectively, {\it semistable}) if for every 
subbundle $0\neq F\subsetneq E$ equipped with the induced parabolic structure, we have 
\[ \mbox{par-}\mu(F) \,<\, \mbox{par-}\mu(E)\ \
\text{(respectively,}\ \mbox{par-}\mu(F) \,\leq\,\mbox{par-}\mu(E)\text{)}\, .\]

\subsection{Poincar\'e parabolic bundle}

Fix integers $r>1$ and $d$ and for each $i=1,2,\cdots, n$ a sequence of positive 
integers $\{r_j^i\}_{j=1}^{k_i}$ such that $\sum_{j=1}^{k_i}r_j^i =
r$ for each $i$. Then the coarse moduli space $M_{X}(d,r, \{\alpha_j^i\},\{r_j^i\})$ of 
semistable parabolic vector bundles of rank $r$, degree $d$, flag types 
$\{r_j^i\}$ and parabolic weights $\{\alpha_j^i\}_{j=1}^{k_i}$ at $x_i \,\in\, D$, $1 
\,\leq\, i \,\leq\, n$, is a normal projective variety \cite{MS}. The open subvariety 
$M_{X}(d,r, \{\alpha_j^i\},\{r_j^i\})^s$ of it consisting of stable parabolic bundles is 
smooth.

For a scheme $S$, let $\pi_X\,:\, X\times S\,\longrightarrow\, X$ and $\pi_S
\,:\, X\times S\,\longrightarrow\, S$ be the natural projections. 
For a vector bundle $U$ over $X \times S$, and $s\,\in\, S$, set $U_s \,:= \, U|_{X \times \{s\}}$.
Given a flag type $m_i=(r_1,\cdots,r_{k_i})$, $1\, \leq\, i\, \leq\, n$, with $\sum_{j=1}^{{k_i}}r_j\,=\,r$, define $\mathcal{F}_{m_i}$ to be the 
variety of flags of type $m_i$. Furthermore, for a vector bundle $U\,\longrightarrow\, S$ of rank
$r$, let $\mathcal{F}_{m_i}(U)\,\longrightarrow\, S$ be the bundle of flags of type $m_i$.

For each $x_i\,\in\, D$ we fix the flag type $m_i\,=\,(r_1^i, r_2^i, \cdots, r_{k_i}^i)$. A {\it family} 
of quasi-parabolic vector bundles parametrized by a scheme $S$ is defined to be a vector bundle 
$U$ over $X\times S$ together with sections $\phi_{x_i}\,:\,S \,\longrightarrow\,
\mathcal{F}_{m_i}(U|_{x_{i}\times 
S})$, $1\,\leq\, i \,\leq\, n$. Note that the section $\phi_{x_i}$ corresponds to a flag of subbundles of 
$U|_{x_{i}\times S}$ with flag type $m_i$ for each $i$. A family of parabolic bundles is given 
by associating weights $\{\alpha_j^i\}$ to each flag of subbundles over $x_{i}\times S$, $x_i\in 
D$. We denote the family of parabolic bundles by $U_*=(U,\phi,\alpha)$ and by $U_{s,*}$ the 
parabolic bundle $(U_s,\phi_s,\alpha)$ above $s \in S$.

It is known that if the elements of the set $\{d,r_j^i\mid 1\leq i \leq n, 1 \leq j \leq k_i\}$ 
have greatest common divisor equal to one then $M_{\alpha}^{s}:=M_{X}(d,r, 
\{\alpha_j^i\},\{r_j^i\})^s$ is a fine moduli space, meaning there exists a family 
$\mathcal{U}_{*}^{\alpha}:=(\mathcal{U},\phi,\alpha)$ parametrized by $M_{\alpha}^{s}$ with the 
property that $\mathcal{U}_{e,*}^{{\alpha}}$ is a stable parabolic bundle isomorphic to $E_*$ 
for all $[E_*]\,=\,e\,\in\, M_{\alpha}^{s}$ \cite[Proposition 3.2]{boden},
\cite{BD}. Moreover, if the 
parabolic weights $\{\alpha_j^i\}$ are chosen to be {\it generic}, i.e.,
the notions of stability and 
semi-stability coincide, then the moduli space $M_{X}(d,r, \{\alpha_j^i\},\{r_j^i\})$ is a 
smooth, irreducible, projective variety. We denote this variety by $M_{\alpha}$.

Now assume that the weights are generic, $\alpha_j^i$ are rational numbers and $r_j^i=1$, so we 
are choosing full flags at each points of $D$. Note that this is the generic case.
There is a well defined {\it determinant 
morphism} $\det\,:\,M_{\alpha}\,\longrightarrow\, J^d(X)$, where $J^d(X)$ denotes the component
of the Picard group of $X$ consisting of line bundles
of degree $d$. For $\Lambda\in J^d(X)$, denote the fiber $\det^{-1}(\Lambda)$ by 
$M_{\alpha}(\Lambda)$, and the restriction of the vector bundle $\mathcal{U}$ to $X\times 
M_{\alpha}(\Lambda)$ by (with a mild abuse of notation) $\mathcal{U}$. From the earlier discussions
it is clear that the vector bundle $\mathcal{U}$ over $X\times 
M_{\alpha}(\Lambda)$ gets a natural parabolic structure over the smooth divisor $D\times 
M_{\alpha}(\Lambda)$.

\subsection{Strongly Parabolic Higgs Fields}

In this subsection we will briefly recall some properties of strongly parabolic Higgs fields and 
the Hitchin map; for details see \cite[Sections 2,3]{GL}. As before, we assume that the 
weights are generic and full flags at each points of $D$.

Let $K_X$ denotes the holomorphic cotangent bundle of $X$.

A {\it parabolic Higgs field} on a parabolic vector bundle $E_*$ is a 
homomorphism $$\Phi\,:\, E \,\longrightarrow\, E\otimes K\otimes{\mathcal O}_X(D)\,=\,
E\otimes K(D)$$
such that
\begin{enumerate}
\item $\text{trace}(\Phi)\,=\, 0$, and

\item $\Phi$ is a {\it strongly parabolic} homomorphism, meaning for each $x_i \in D$ we have
$\Phi(F_i^j) \,\subset\, F_i^{j+1} \otimes (K(D)|_{x_i})$.
\end{enumerate}
The pair $(E_*,\, \Phi)$ is called a parabolic
Higgs bundle.

A parabolic Higgs bundle $(E_*, \, \Phi)$ is called {\it stable} (respectively,
{\it semistable}) if for all proper non-zero $\Phi$-invariant 
sub-bundles $F$ of $E$, we have $\mbox{par-}\mu(F) \,<\, \mbox{par-}\mu(E)$
(respectively, $\mbox{par-}\mu(F) \,\leq\, \mbox{par-}\mu(E)$).
The cotangent space at $[E]\,\in\, M_{\alpha}(\Lambda)$ can be identified 
with $H^0(X, \mbox{SParEnd}_0(E)\otimes K_X(D))$ where $\mbox{SParEnd}_0(E)$ is the sheaf of 
strongly parabolic traceless endomorphisms. Then the coefficients of the characteristic polynomial of 
$\phi\,\in\, H^0(X, \mbox{SParEnd}_0(E)\otimes K_X(D))$ lie in $W\,:=\,\bigoplus_{j=2}^{r}
H^0(X,\, K_{X}^j((j-1)D))$.

Let $N_{\alpha}(\Lambda)$ be the moduli space of isomorphism classes of strongly parabolic 
stable Higgs bundles with parabolic structures over $D$ and weights $\{\alpha_j^i\}_{j=1}^{r}$ 
at $x_i \,\in\, D$, $1 \,\leq\, i \,\leq\, n$, with fixed determinant $\Lambda$. The total space 
$T^{*}M_{\alpha}(\Lambda)$ of the cotangent bundle is an open subvariety of the moduli space 
$N_{\alpha}(\Lambda)$. The map $$h\,:\,N_{\alpha}(\Lambda)\,\longrightarrow\,
W\, , \ \ (E,\,\phi)\,\longmapsto \, ({\rm trace}(\wedge^2\phi),\cdots ,{\rm trace}(\wedge^r\phi))$$
is proper and surjective; it is called the {\it Hitchin 
map}. If $s\,\in\, W$ such 
that the corresponding spectral curve $X_{s}$ is smooth, then the fiber $h^{-1}(s)$
is identified with the Prym variety 
$$\mbox{Prym}^{\delta}(X_s)\,=\,\{L\,\in\, J^{\delta}(X_s)\,\mid\,
\det(\pi_*L)\simeq \Lambda\}$$ 
associated to $X_{s}$, where $\delta\,:=\,d-\mbox{deg}(\pi_*(\mathcal{O}_{X_s}))$.

A parabolic bundle on $X$ is called {\it very stable} if there is no non-trivial nilpotent strongly 
parabolic Higgs field on it. It is known that, if the genus $g(X)\geq 2$, a very stable parabolic bundle is stable. There 
exist very stable parabolic bundles in any moduli space. In fact the subset of very stable parabolic 
bundles is a dense open set in $M_{\alpha}(\Lambda)$. This follows from the fact the dimension 
of the nilpotent cone $h^{-1}(0)$ is same as the dimension of the moduli space 
$M_{\alpha}(\Lambda)$ \cite[Corollary 3.10]{gmg}. Let
$$S'\, \subset\, S\, :=\, T^*M_{\alpha}(\Lambda)\cap
h^{-1}(0)$$ be the open subset consisting of all $(E,\phi)\,\in\, S$ such
that $\phi$ is nonzero. The image of $S'$ in $M_{\alpha}(\Lambda)$ under the forgetful map
$(E,\phi)\,\longmapsto\, E$ will be denoted by $B$. Note that $B$ is the non-very stable
locus in $M_{\alpha}(\Lambda)$. On the other hand, there is a free action of
${\mathbb C}^*$  on $S'$; namely the action of any $c\,\in\, {\mathbb C}^*$ sends any
$(E,\phi)$ to $(E,c\cdot\phi)$. Hence we have
$$
\dim M_{\alpha}(\Lambda) \,=\, \dim S\, =\, \dim S' \, >\, \dim B\, .
$$
This implies that the complement $M_{\alpha}(\Lambda) \setminus B$ is nonempty.

\subsection{Determinant bundle}

Let $T$ be a variety. For any coherent sheaf $\mathcal{E}$ on $X\times T$, flat over $T$, let 
$\det R\pi_{T}\mathcal{E}$ denote {\it determinant line bundle} defined as:
\[\{\det R\pi_{T}\mathcal{E}\}_t :=\{\det H^0(X,\mathcal{E}_t)\}^{-1}\otimes \{\det H^1(X,\mathcal{E}_t)\}\]
for $t\,\in\, T$ (\cite{BR}, \cite{nara}, \cite{Sunn}).

Let $x$ be a fixed closed point of $X$. We fix rational numbers $0\leq \alpha_1<\alpha_2<\cdots< \alpha_r<1$ 
and a positive integer 
$k$ such that $\beta_j=k\cdot \alpha_j$ is an integer for each $j=1, \cdots, r$. 
Set $d_j:= \beta_{j+1}-\beta_j, 1\leq j \leq r$, with the assumption that
$\beta_{r+1}\,=\,1$. 
Let $\mathcal{V}_*^{\alpha}=(\mathcal{V}^{\alpha}, \alpha=\{\alpha_j\}_{j=1}^{r},\phi)$ be a family of 
rank $r$ stable parabolic bundles over $X$ with parabolic divisor $\{x\}$ parametrized by a variety 
$T$ and
\[\mathcal{V}^{\alpha}|_{x\times T}=\mathcal{F}_{1,x}\supset \mathcal{F}_{2,x}\supset \cdots \supset 
\mathcal{F}_{r,x}\supset \mathcal{F}_{r+1,x}=(0)\] 
be the full flag of subbundles over $x\times T$ determined by the section $\phi$. 
Set $L_j := \frac{\mathcal{F}_{j,x}}{\mathcal{F}_{j+1,x}}$. Let $\Psi: T \longrightarrow M_{\alpha}(\Lambda)$ 
be the morphism induced by this family. Define a line bundle
\[\theta_{T} :=(\det R\pi_{T}\mathcal{V}^{\alpha})^{k}\otimes \det (\mathcal{V}_x^{\alpha})^l
\otimes \otimes_{j=1}^{r} L_j^{d_j}\]
where $l$ is a positive integer determined by \cite[Equation (*), page 6]{Sunn}. Then there exists a unique 
(up to algebraic equivalence) ample line bundle $\Theta_{M_{\alpha}}$ over $M_{\alpha}(\Lambda)$ such that
$\Psi^*\Theta_{M_{\alpha}}=\theta_{T}$ \cite{BR}, \cite[Theorem 1.2]{Sunn}.

\section{parabolic stability of the parabolic Poincar\'e bundle}

In this section we continue with the notation of the previous section.

Let $X$ be a smooth projective irreducible complex curve of genus $g \geq 2$, 
$x_1, \cdots, x_n \,\in\, X$ distinct points and $D=x_1+\cdots+x_n$. Let $Y$ be a 
smooth projective irreducible complex variety.

For each point $x_{i}\in \text{Supp}(D)$ 
fix real numbers $0\leq \alpha_1^i < \alpha_2^i < \cdots <\alpha_{k_i}^i < 1$
and $m_i=(r_1^i,\cdots, r_{k_i}^i)$, where each $r_j^i$ is a positive integer. Let $U$ be a rank $r$ vector 
bundle over $X\times Y$ with parabolic structure over the smooth divisor $D\times Y$, of flag types $m_i$ 
and weights $\{\alpha_j^i\}, 1 \leq i\leq n, 1 \leq j \leq k_i$. Fix ample divisors $\theta_{X}$ on $X$ and 
$\theta_{Y}$ on $Y$. Then for any integers $a,b>0$, the class $a\theta_{X}+b\theta_{Y}$ is ample on $X\times Y$. 
Let $\Theta := a\theta_{X}+b\theta_{Y}$ for some fix integers $a,b>0$.

\begin{lemma}\label{two stable}
Suppose that for a general point $x\in X,$ the vector
bundle $U_x\,=\, U\vert_{\{x\}\times Y}$ is semi-stable with respect to $\theta_Y$
over $Y$, and 
for a general point $y\,\in\, Y$ the parabolic vector bundle $U_y\,=\,
U\vert_{X\times\{y\}}$ over $X$ with parabolic divisor $D$ is semi-stable with respect 
to $\theta_X$. Then the parabolic vector bundle $U$ with
parabolic divisor $D\times Y$ is parabolic semi-stable with respect to 
$\Theta$. Moreover, if $U_x$ is stable or $U_y$ is parabolic stable, then
$U$ is also parabolic stable.
\end{lemma}

\begin{proof}
The proof essentially follows from the proof of \cite[Lemma 2.2]{BBN}. Let us 
indicate the modification needed in this case. Let $F\,\subset\, U$ be a torsionfree 
subsheaf. Then it has an induced parabolic structure. To compute the parabolic degrees
of $U$ and $F$, one needs to 
compute the degree of certain vector bundles supported on the smooth divisor 
$D\times Y$. But this is same as computing the degree of certain subsheaves of $U$ 
and $F$, which can be done as in \cite[Lemma 2.2]{BBN}.
\end{proof}

For the rest of this section we assume that the parabolic weights are rational, generic and full 
flags at each points of $\text{Supp}(D)$.

Set $s \,=\, (s_2,\cdots,s_r)\in W$, and let $\pi:X_s \longrightarrow X$ be the associated spectral cover. 
Then for $z\in X$, the fiber $\pi^{-1}(z)$ is given by the points $y\,\in\,
K_{X}(D)|_{z}$ which satisfy the polynomial 
 \[y^r+s_{2}(z)\cdot y^{r-2}+\cdots +s_r(z)\,=\,0\, .\] 
 Let us denote this polynomial by $f$. The morphism $\pi$ is unramified over $z$ if
and only if the resultant $R(f,f')$ of $f$ and its derivative $f'$ are nonzero. Since
all $s_j$ vanish over $D$, the ramification locus of $\pi$ contains $D$. 
 
\begin{lemma}\label{unram}
 Let $X$ be a smooth, irreducible, projective curve of genus $g\geq 2$ and $z\notin \mbox{Supp}(D)$. There exists a smooth, projective spectral curve $Y$ and finite morphism 
 $\pi:Y\longrightarrow X$ of degree $r$ which is unramified over $z$.
\end{lemma}

\begin{proof}
Since the linear system $|K_X^{j}((j-1)D)|$ is base point free outside $D$ and $z\notin \mbox{Supp}(D)$, 
there exists $(s_2,\cdots s_r)\in W$ such that \[ R(f,f')(s_2(z),\cdots s_r(z))\neq 0.\]
Clearly this is an open condition in $W$. Thus there exists a non-empty 
open subset $V$ of $W$ such that for each $s \in V$, the corresponding spectral cover 
$X_s \longrightarrow X$ is unramified over $z$. Now, since the genus $g\geq 2$, by \cite[Lemma 3.1]{GL} 
the set of points in $W$ where the corresponding spectral curve is smooth is an dense open subset of $W$. 
Thus we can always choose a spectral curve which is smooth and unramified over $z$.
\end{proof}

\begin{lemma}
 Let $X_s \longrightarrow X$ be a spectral curve. Let $P^{\delta}$ be the associated Prym
variety, where $\delta\,:=\,d-\mbox{deg}(\pi_*(\mathcal{O}_{X_s}))$. Then there is a 
dominant rational map $f\,:\, P^{\delta}\,\dashrightarrow \,M_{\alpha}(\Lambda)$.
\end{lemma}

\begin{proof}
 Let $h'$ be the restriction of $h$ to the the total space of the cotangent bundle $T^{*}M_{\alpha}(\Lambda)$. 
 Then for any very stable parabolic bundle $E \,\in\, M_{\alpha}(\Lambda)$, the
restriction 
$$h'_E\,:\,T_E^*M_{\alpha}(\Lambda)\,\longrightarrow\, W$$ of $h'$ is surjective (for a proof see \cite[Lemma 1.4]{KP}). 
 Thus, for any $s\in W$, we have $h'^{-1}(s)\cap T_E^{*}M_{\alpha}(\Lambda)$ 
 is nonempty for every very stable parabolic bundle $E \,\in\, M_{\alpha}(\Lambda)$. 
 Consequently, for all $s\in W$, the image of the map $h'^{-1}(s)\longrightarrow M_{\alpha}(\Lambda)$ 
 contains the dense open set $U$ of all very stable parabolic bundles. Thus the morphism 
 $h'^{-1}(s) \longrightarrow M_{\alpha}(\Lambda)$ is dominant. 
 Since $h'^{-1}(s)\subseteq h^{-1}(s)\simeq P^{\delta}$ is an open set, we have a dominant rational 
 map $f\,:\, P^{\delta} \,\dashrightarrow\, M_{\alpha}(\Lambda)$.
\end{proof}

Now we discuss the `parabolic stability' of $\mathcal{U}$ with respect to a `naturally' defined ample 
divisor on $X\times M_{\alpha}(\Lambda)$. For the simplicity of the exposition we assume that $D=x$ 
(for an arbitrary reduced divisor the same arguments will hold).

\begin{theorem}\label{th1}
Let $z\notin \text{Supp}(D)$. Then $\mathcal{U}_z$ is semi-stable with respect
the ample divisor $\Theta_{M_{\alpha}}$.
\end{theorem}
 
\begin{proof}
By Lemma \ref{unram} we get a spectral cover $\pi: Y\longrightarrow X$ which is
unramified over $z$. Let $\pi^{-1}(z)\,=\,\{y_1,\,\cdots,\,y_{r}\}$, with $y_i$
being distinct points in $Y$.

Let $\pi\times 1: Y\times P^{\delta}\longrightarrow X\times P^{\delta}$ denotes 
the product morphism. Let $\mathcal{L}$ denote the restriction of a Poincar\'e 
line bundle on $Y\times J^{\delta}(Y)$ to $Y\times P^{\delta}$. Then the
direct image $(\pi\times 
1)_{*}\mathcal{L}$ is a rank $r$ vector bundle and the $\mathcal{O}_{X\times 
P^{\delta}}$--algebra structure on $(\pi\times 1)_{*}\mathcal{L}$ defines a 
section $$\Phi\,\in\, H^0(X\times P^{\delta},\,\mbox{End}((\pi\times 
1)_{*}(\mathcal{L}))\otimes p_{X}^*K_{X}(D))\, .$$ This
$\Phi$ induces a parabolic 
structure on $(\pi\times 1)_{*}(\mathcal{L})$ over $x\times P^{\delta}$. Thus we 
have a family of parabolic bundles parametrized by $P^{\delta}$. Clearly, the 
rational map $f:P^{\delta}\dashrightarrow M_{\alpha}(\Lambda)$ is induced by the 
above family. Let $T^{\delta}$ be the open set where $f$ is defined. Then 
$\mbox{Codim}(P^{\delta}\setminus T^{\delta})\geq 2$.

Let $\mathcal{E}:=((\pi\times 1)_*\mathcal{L})|_{X\times T^{\delta}}$. Since 
$M_{\alpha}(\Lambda)$ is a fine moduli space we have
\[
(1\times f)^*\mathcal{U} 
\simeq \mathcal{E}\otimes p_{_{T^{\delta}}}^*(L_0)
\]
for some line bundle $L_{0}$ on $T^{\delta}$. Thus \[f^*\mathcal{U}_z \simeq 
\oplus_{i=1}^{r}\mathcal{L}_{y_i}\otimes L_0\] on $T^{\delta}$. Since 
$\mbox{Codim}(P^{\delta}\setminus T^{\delta})\geq 2$ and $P^{\delta}$ is smooth,
the line bundles $\mathcal{L}_{y_i}$ and $L_0$ uniquely extend over $P^{\delta}$. 
The line bundles $\mathcal{L}_{y_i}$ are already defined over $P^{\delta}$. Let 
$L'_0$ be the unique extension of $L_0$ over $P^{\delta}$. Since 
$\mathcal{L}_{y_i}$ are algebraically equivalent, it follows that
$\oplus_{j=1}^{r}\mathcal{L}_{y_i}\otimes L'_0$ is semistable with respect to any 
ample line bundle on $P^{\delta}$. Thus if we can find an ample line bundle $H$ 
over $P^{\delta}$ such that $H|_{T^{\delta}} \simeq f^*(\Theta_{M_{\alpha}}^n)$ 
for some positive integer $n$, then by \cite[Lemma 2.1]{BBN}, $\mathcal{U}_{z}$ is 
semistable with respect to $\Theta_{M_{\alpha}}^n$. Hence it is semistable with 
respect $\Theta_{M_{\alpha}}.$

We have, 
\[f^*\Theta_{M_{\alpha}}=\theta_{T^{\delta}}\,=\,
(\det R{\pi_{T^{\delta}}}\mathcal{E})^{k}\otimes 
\det(\mathcal{E}_x)^{l}\otimes \otimes_{j=1}^{r}L_j^{d_j}\, .
\]
By \cite[Theorem 4.3]{Li} we get that 
\[(\det R{\pi_{T^{\delta}}}\mathcal{E})^{k}\,=\,
m\Theta_{P^{\delta}}^{k}|_{_{T^{\delta}}}
\] 
for some positive integer $m$, where $\Theta_{P^{\delta}}^{k}$ is the restriction 
of the canonical theta divisor on $J^{\delta}(Y)$ to $P^{\delta}$. Let $M$ be the 
unique extension of $\det(\mathcal{E}_x)^l\otimes \otimes_{j=1}^{r}L_j^{d_j}$. Set 
$H:= m\Theta_{P^{\delta}}^{k}\otimes M$. Then for some positive integer $q, H^q$ is 
ample on $P^{\delta}$. Thus 
$f^*\Theta_{M_{\alpha}}^n$ is a restriction of an ample line bundle $H$ on 
$P^{\delta}$.
\end{proof}

As a corollary of Theorem \ref{th1} and Lemma \ref{two stable} we obtain the main 
result:

\begin{theorem}
 The parabolic bundle $\mathcal{U}$ over $X\times M_{\alpha}(\Lambda)$ is parabolic stable with respect 
 to any integral ample divisor of the form $aD_X+b\Theta_{M_{\alpha}}$, where $D_X$ is an ample divisor 
 on $X$ and $a,b>0$.
\end{theorem}

\section*{Acknowledgements}

The third-named author is supported by NBHM Post-doctoral Fellowship, DAE 
(Government of India). The second-named author is supported by a J. C. Bose
Fellowship.

\end{document}